%% file: main.tex
\documentclass[9pt,shortpaper,twoside,web]{ieeetran}

\usepackage{cite}
\usepackage{amsmath,amssymb,amsfonts}
\usepackage{graphicx}
\usepackage{textcomp}

\def\BibTeX{{\rm B\kern-.05em{\sc i\kern-.025em b}\kern-.08em
		T\kern-.1667em\lower.7ex\hbox{E}\kern-.125emX}}

\usepackage{enumitem}

\usepackage[caption=false]{subfig}

\usepackage{balance}
\usepackage{xfrac}
\usepackage{xspace}
\usepackage{bbm}		
\usepackage{array}
\usepackage{algorithm}
\usepackage{algpseudocode}
\usepackage{multirow}
\usepackage{booktabs}
\usepackage{footnote}
\usepackage{threeparttable}
\usepackage{makecell}
\usepackage{cellspace}
\usepackage{amsthm}

\usepackage[normalem]{ulem}

\usepackage{mathtools}
\usepackage{newfloat}

\DeclarePairedDelimiter{\abs}{\lvert}{\rvert}
\DeclarePairedDelimiter{\norm}{\lVert}{\rVert}

\DeclarePairedDelimiter{\Fronorm}{\lVert}{\rVert_{F}}
\DeclarePairedDelimiter{\Twonorm}{\lVert}{\rVert_{2}}

\newcommand{\invnb}[1]{#1^{-1}}

\newcommand{\trans}[1]{#1'}
\newcommand{\ctrans}[1]{#1^{*}}

\newcommand{\Linf}{\mathcal{L}^{\infty}}
\newcommand{\Cpi}{\mathcal{\tilde{C}}_{2\pi}}

\newtheorem{proposition}{Proposition}

\newtheorem{theorem}{Theorem}
\newtheorem{lemma}{Lemma}
\newtheorem{corollary}{Corollary}

\newtheorem{remark}{Remark}
\newtheorem{definition}{Definition}

\usepackage{tikz}
\usepackage{pgf}
\usepackage{pgfplots}
\usetikzlibrary{positioning, calc}
\pgfplotsset{compat=1.13}
\usepgfplotslibrary{groupplots}

\begin{document}

\title{Horizon-independent Preconditioner Design for Linear Predictive Control}

\author{Ian~McInerney~\IEEEmembership{Student~Member,~IEEE}, Eric~C.~Kerrigan~\IEEEmembership{Senior~Member,~IEEE}, George~A.~Constantinides~\IEEEmembership{Senior~Member,~IEEE}
	\thanks{The support of the EPSRC Centre for Doctoral Training in High Performance Embedded and Distributed Systems (HiPEDS, Grant Reference EP/L016796/1) is gratefully acknowledged.}%
	\thanks{The authors are with the Department of Electrical \& Electronic Engineering, Imperial College London, SW7~2AZ, U.K. E.C.\ Kerrigan is also with the  Department of Aeronautics. email: \{i.mcinerney17, e.kerrigan, g.constantinides\}@imperial.ac.uk}}

\maketitle

\begin{abstract}
 First-order optimization solvers, such as the Fast Gradient Method, are increasingly being used to solve Model Predictive Control problems in resource-constrained environments.
 Unfortunately, the convergence rate of these solvers is significantly affected by the conditioning of the problem data, with ill-conditioned problems requiring a large number of iterations.
 To reduce the number of iterations required, we present a simple method for computing a horizon-independent preconditioning matrix for the Hessian of the condensed problem.
 The preconditioner is based on the block Toeplitz structure of the Hessian.
 Horizon-independence allows one to use only the predicted system and cost matrices to compute the preconditioner, instead of the full Hessian.
 The proposed preconditioner has equivalent performance to an optimal preconditioner in numerical examples, producing speedups between 2x and 9x for the Fast Gradient Method.
 Additionally, we derive horizon-independent spectral bounds for the Hessian in terms of the transfer function of the predicted system, and show how these can be used to compute a novel horizon-independent bound on the condition number for the preconditioned Hessian.
\end{abstract}


\section{Introduction}

Model Predictive Control (MPC) is an optimal control method that aims to optimize the closed-loop performance of a controlled system by solving an optimization problem at each sampling instant to compute the next control input.
 MPC is swiftly becoming a popular choice for the control of complicated systems with operational constraints, due to its explicit handling of constraints and recent advances in real-time optimization algorithms, allowing it to be deployed on resource-constrained systems such as internet-of-things devices \cite{Lucia2016_CPSperspectives}.

A common formulation of MPC is the Constrained Linear Quadratic Regulator (CLQR), which is an extension of the LQR to handle state and input constraints for a linear system with a quadratic objective function.
 The CLQR formulation gives an optimization problem that is a convex Quadratic Program (QP), which can then be solved efficiently using several types of iterative methods, such as interior-point methods, active-set methods and first-order methods.
 These methods are all affected by the conditioning of the problem, with poorly conditioned problems requiring more iterations to find the optimal solution.
 To overcome the problem's ill-conditioning and reduce the number of iterations, implementations utilize preconditioning techniques on the problem data for interior-point methods \cite{malyshevPreconditioningConjugateGradient2018}, active-set methods \cite{quirynenBlockStructuredPreconditioning2018}, and first-order methods \cite{Giselsson2014_precondFastDual, Richter2012_FGMcomplexity}.

We focus on first-order methods, which commonly use preconditioners generated by solving semidefinite programs (SDPs).
 One such SDP formulation is \cite{Richter2012_FGMcomplexity}, which minimizes the maximal eigenvalue of the preconditioned matrix by embedding the Hessian into a linear matrix inequality constraint.
 This SDP is readily solvable, but embedding the full Hessian into the constraints means the preconditioner must be recomputed if the horizon length changes and that the SDP problem size grows with the horizon length --- making the SDP preconditioner time consuming to compute for systems with long horizons and slowing down the control design process.

In this work, we present a new preconditioner for the condensed CLQR formulation, and show its ability to speed up convergence of the Fast Gradient Method (FGM) by up to 9x.
 This preconditioner is horizon-independent and is computed using only matrices with the number of states and inputs as their dimensions, but requires the terminal penalty matrix $P$ to be either the solution to the discrete Lyapunov equation for Schur-stable systems, or the discrete Ricatti equation for unstable systems.
 We show that our proposed preconditioner provides performance equivalent to an existing SDP preconditioner on several examples, while also providing a reduction in the computational effort required to compute the preconditioner.
 Additionally, we exploit the block Toeplitz structure of the CLQR problem's condensed Hessian to derive tight horizon-independent bounds on the extremal eigenvalues and condition number of both the original and preconditioned Hessians.

We begin in Section~\ref{sec:prelim} by introducing the CLQR problem formulation and the notation used in this paper.
 In Section~\ref{sec:spectralProperties}, we derive the theoretical framework for the computation of the bounds on the extremal eigenvalues and condition number of the Hessian.
 We propose our new preconditioner and extend the bounds to the preconditioned Hessian in Section~\ref{sec:precond}.
 Finally, we present numerical examples in Section~\ref{sec:comparisons} that compare the proposed preconditioner against existing preconditioners, and show its effect on the convergence rate of the Fast Gradient Method applied to three different systems.

\section{Preliminaries}
\label{sec:prelim}

\subsection{Notation}

$\trans{A}$ and $\ctrans{A}$ denote the transpose and conjugate-transpose of the matrix $A$, respectively.
 $A\otimes B$ represents the Kronecker product of matrix $A$ with matrix $B$.
 $\lambda_1 \leq \dots \leq \lambda_k$ are the real eigenvalues of a Hermitian matrix $A$ in sorted order, with the set of all eigenvalues denoted by $\lambda(A)$.
 The $p$-norm is denoted by $\norm{\circ}_p$, with $\Twonorm{A}$ the matrix spectral norm, and $\Fronorm{A}$ the Frobenius norm.
 The condition number of a matrix is $\kappa(A)\coloneqq\nobreak\Twonorm{A}\Twonorm{A^{-1}}$.
 The set $\mathbb{T} \coloneqq \{z \in \mathbb{C} : \abs{z} = 1 \}$ is the complex unit circle.
 For an infinite-dimensional block Toeplitz matrix $\mathbf{T}$ with blocks of size $m \times n$, $\mathcal{P}_{T}(\cdot) : \mathbb{T} \to \mathbb{C}^{m \times n}$ represents its matrix symbol and $T_{N}$ represents the truncated version of $\mathbf{T}$ after $N$ block diagonals (where $N$ is a positive integer).

We use the notation from \cite{zhouRobustOptimalControl1996} to represent the transfer function matrix of the system $\mathcal{G}$ with state space matrices $A, B, C, D$ as
 \begin{equation*}
   	\mathcal{G}(z) \coloneqq
   	\left[
   	\begin{array}{c|c}
   	A & B \\
   	\hline
   	C & D
   	\end{array}
	\right].
 \end{equation*}

$\Linf$ is the space of matrix-valued essentially bounded functions (i.e.\ matrix-valued functions that are measurable and have a finite Frobenius norm almost everywhere on their domain, see \cite[\S 2]{Miranda2000}).
 $\Cpi$ is the space of continuous $2\pi$-periodic functions inside $\Linf$.

\begin{definition}
	Let $\mathcal{P}_{T}(\cdot) \in \Cpi$ be a function that maps $\mathbb{T} \to \mathbb{C}^{n \times n}$, we define the extreme eigenvalues of $\mathcal{P}_{T}(\cdot)$ as
	\begin{align*}
	\lambda_{min}(\mathcal{P}_{T}) \coloneqq \underset{z \in \mathbb{T}}{\inf}~\lambda_1(\mathcal{P}_{T}(z)), \quad
	\lambda_{max}(\mathcal{P}_{T}) \coloneqq \underset{z \in \mathbb{T}}{\sup}~\lambda_n(\mathcal{P}_{T}(z)),
	\end{align*}
	and the condition number of $\mathcal{P}_{T}(\cdot)$ as
    \begin{equation*}
        \kappa(\mathcal{P}_{T}) \coloneqq \frac{\lambda_{max}(\mathcal{P}_{T})}{\lambda_{min}(\mathcal{P}_T)}.
    \end{equation*}
\end{definition}

\begin{definition}
	Let $T_{n}$ be the $n{\times}n$ truncation of the infinite matrix~$\mathbf{T}$. If the limits exist, we define the extrema of the spectrum of $\mathbf{T}$ as
	\begin{align*}
	\lambda_{min}(\mathbf{T}) &\coloneqq \lim_{n \to \infty}~\lambda_1(T_{n}), \quad
	\lambda_{max}(\mathbf{T}) \coloneqq \lim_{n \to \infty}~\lambda_n(T_{n}).
	\end{align*}
\end{definition}

\subsection{CLQR Formulation}
\label{sec:prelim:formulation}

In this work, we examine the Constrained Linear Quadratic Regulator (CLQR) formulation of the MPC problem, which can be written as the constrained QP
\begin{subequations}
    \label{eq:mpc:linMPC}
    \begin{align}
    \underset{u,x}{\text{min}}\   & \frac{1}{2} \trans{x_N} P x_N + \frac{1}{2} \sum_{k=0}^{N-1}
    \trans{ \begin{bmatrix} x_k\\ u_k \end{bmatrix} }
    \begin{bmatrix}
    Q & 0\\
    0 & R
    \end{bmatrix}
    \begin{bmatrix} x_k\\ u_k \end{bmatrix}
    \label{eq:mpc:lin:cost}\\
    \text{s.t.\ }  &
    \begin{aligned}[t]
    x_{k+1} &= A x_k + B u_k,\ k=0, \ldots, N-1 \label{eq:mpc:lin:dyn} \\
    x_{0} &= \hat{x}
    \end{aligned}\\
    & E_{u} u_k + E_{x} x_k\leq c,\ k=0, \ldots, N-1 \label{eq:mpc:lin:ineqCon}
    \end{align}
\end{subequations}
 where $N$ is the horizon length, $x_k \in \mathbb{R}^{n}$ are the states, and $u_k \in \mathbb{R}^{m}$ are the inputs at sample instant $k,$ and $\hat{x} \in \mathbb{R}^{n}$ is the current measured system state.
 $A \in \mathbb{R}^{n \times n}$ and $B \in \mathbb{R}^{n \times m}$ are the state-space matrices describing the discrete-time system with transfer function matrix
 \begin{equation*}
   	\mathcal{G}(z) \coloneqq
   	\left[
   	\begin{array}{c|c}
   	A & B \\
   	\hline
   	I & 0
   	\end{array}
	\right].
 \end{equation*}%
 $E_{u} \in \mathbb{R}^{l \times m}$ and $E_{x} \in \mathbb{R}^{l \times n}$ are the stage constraint coefficient matrices, and $c \in \mathbb{R}^{l}$ is the vector of bounds for the stage constraints.
 The matrices $Q = \trans{Q} \in \mathbb{R}^{n \times n}, R = \trans{R} \in \mathbb{R}^{m \times m}$ and $P=\trans{P} \in \mathbb{R}^{n \times n}$ are the weighting matrices for the system states, inputs and final states, respectively.
 The weighting matrices are chosen such that $P$, $Q$ and $R$ are positive definite.

This problem can be condensed by removing the state variables from~\eqref{eq:mpc:linMPC} to leave only the control inputs in the vector
$ u {\coloneqq} \trans{\begin{bmatrix}
\trans{u_0}~\trans{u_1}~\cdots~\trans{u_{N-1}}
\end{bmatrix}}$.
 The optimization problem is then
 \begin{subequations}
    \label{eq:mpc:condMPC}
    \begin{align}
    \underset{u}{\text{min}}\   & \frac{1}{2} \trans{u} H u + \trans{\hat{x}} \trans{\Phi} u\label{eq:mpc:cond:cost}\\
    \text{s.t.\ }  & G u \leq F \hat{x} + g \label{eq:mpc:cond:con}
    \end{align}
 \end{subequations}
 with
 $ H \coloneqq \trans{\Gamma} \bar{Q} \Gamma + \bar{R} $,
 $ \bar{R} \coloneqq I_{N} \otimes R$,
 $ \bar{Q} \coloneqq \begin{bmatrix}
     I_{N-1} \otimes Q & 0\\
     0 & P
     \end{bmatrix},
 $
 \begin{equation*}
     \Gamma \coloneqq \begin{bmatrix}
      B & 0 & 0 & & 0\\
      AB & B & 0 & & 0\\
      A^2B & AB & B & & 0 \\
      \vdots & & & \ddots & \vdots\\
      A^{N-1}B & A^{N-2}B & A^{N-3}B & \cdots & B
     \end{bmatrix}.
 \end{equation*}

The terminal weight matrix $P$ in~\eqref{eq:mpc:linMPC} can be crucial to the stability and performance of the closed-loop controller.
 The simplest choice is to set $P = Q$, so that final states are weighted the same as other states.
 This allows for simple formation of the problem matrices, but does not provide stability guarantees for the closed-loop problem.

Instead, a possible choice for $P$ is to either choose it to be the solution to the Discrete Algebraic Riccati Equation (DARE)
 \begin{equation}
	\label{eq:dare}
	P = \trans{A} P A + Q - \trans{A} P B \invnb{ (\trans{B} P B + R) } \trans{B} P A,
 \end{equation}
 or choose $P$ to be the solution to the discrete Lyapunov equation
 \begin{equation}
    \label{eq:dlyap}
    \trans{A} P A + Q = P,
 \end{equation}
 where $Q$ and $R$ are the cost matrices from Problem~\eqref{eq:mpc:linMPC}, and $A$ and~$B$ are the system matrices.
 Choosing $P$ in these ways approximates the cost function value for the time after the horizon ends, and allows for the derivation of closed-loop stability guarantees for the controller~\cite{Mayne2000_StabilitySurvey}.

\subsection{Numerically Robust CLQR Formulation}
\label{sec:numericallyRobust}

The condensed formulation~\eqref{eq:mpc:condMPC} with an unstable system can become numerically unstable as the horizon length increases, since unstable systems have $\abs{\lambda_{max}(A)} > 1$, taking repeated powers of $A$ to form~$\Gamma$ then causes the condition number of $H$ to be unbounded.

To overcome this, a modification to \eqref{eq:mpc:condMPC} was proposed in \cite{Rossiter1998_prestabilization} where instead of using the actual system $A$ for computing the prediction matrix and optimal control, a prestabilized system $A - BK$ is used instead.
 This prestabilized system would guarantee that the prediction matrix entries do not grow unbounded with the horizon length, leading to better conditioning of the optimization problem.
 Note that this transformation requires the pair $(A,B)$ to be stabilizable.

To formulate the prestabilized problem, a new system $\mathcal{G}_{c}$ is formed by setting $u_{k} = -Kx_{k} + v_{k}$ where $v \in \mathbb{R}^{m}$ is a new input so that
\begin{equation}
	\label{eq:system:gk}
   	\mathcal{G}_{c}(z) \coloneqq
   	\left[
   	\begin{array}{c|c}
   	A - BK & B \\
   	\hline
   	I & 0
   	\end{array}
	\right],
 \end{equation}
 where the controller $K$ is chosen so that $\abs{\lambda_{max}(A - BK)} < 1$.
 There are several ways to compute a $K$ that meets this criterion, however we focus on $K$ chosen as the unconstrained infinite-horizon LQR controller computed using $Q$ and $R$ from \eqref{eq:mpc:lin:cost}.

The computations are then done using the input space
$v {\coloneqq} \trans{\begin{bmatrix}
\trans{v_0}~\trans{v_1}~\cdots~\trans{v_{N-1}}
\end{bmatrix}}$,
turning the condensed QP~\eqref{eq:mpc:condMPC} into
 \begin{subequations}
	\label{eq:mpc:stablecondMPC}
	\begin{align}
	v^{*} \coloneqq
	\underset{v}{\text{argmin}}\   & \frac{1}{2} \trans{v} H_{c} v + \trans{\hat{x}} \trans{\Phi}_{c} v\label{eq:mpc:stable:cost}\\
	\text{s.t.\ }  & G_c v \leq F_c \hat{x} + g \label{eq:mpc:stable:con}
	\end{align}
 \end{subequations}
 with
 $ A_{c} = A - BK$,
 $ Q_{c} = Q + \trans{K} R K $,
 $\bar{K} = I_{N} \otimes -K$,
 \begin{equation*}
     H_{c} \coloneqq \trans{\Gamma_{c}} \bar{Q}_{c} \Gamma_{c} + \trans{\Gamma_{c}}\trans{\bar{K}}\bar{R} + \bar{R}\bar{K}\Gamma_{c} + \bar{R},
 \end{equation*}
 $$\bar{Q}_{c} \coloneqq \begin{bmatrix}
   I_{N-1} \otimes Q_{c} & 0\\
   0 & P
   \end{bmatrix},
 $$
 \begin{equation}
	\Gamma_{c} \coloneqq \begin{bmatrix}
	B & 0 & 0 & \cdots & 0\\
	A_{c}B & B & 0 & & 0\\
	A_{c}^2B & A_{c}B & B & & 0 \\
	\vdots & & & \ddots & \vdots\\
	A_{c}^{N-1}B & A_{c}^{N-2}B & A_{c}^{N-3}B & \cdots & B
	\end{bmatrix}.
	\label{eq:mat:stableGamma}
 \end{equation}
 The input applied to the original system is then $u_{0} = -K\hat{x} + v^{*}_{0}(\hat{x})$.
 The constraint matrices $F_{c}$ and $G_{c}$ in Problem~\eqref{eq:mpc:stablecondMPC} are formed by modifying $F$ and $G$ from Problem~\eqref{eq:mpc:condMPC} to use the prestabilized system, with full details given in \cite{Rossiter1998_prestabilization}.

\section{Spectral Properties}
\label{sec:spectralProperties}

In order to effectively analyze and derive our closed-form preconditioner, we first derive some spectral properties for the Hessian and prediction matrix in the CLQR problem.
 Similar results were reported in \cite{Rojas2004_HessAsymp} and \cite[Sect.~11]{Goodwin2005}, but our analysis differs in several ways.
 Specifically, the prior work requires that the system $\mathcal{G}$ have no eigenvalues on the unit circle, and that it be transformed into a system with separable stable and unstable modes so that they can be handled separately.
 Our analysis uses the numerically robust formulation~\eqref{eq:mpc:stablecondMPC} to support eigenvalues  of $\mathcal{G}$ on the unit circle and both its stable and unstable modes at the same time.
 Additionally, the prior work assumes $Q = \trans{C}C$ (with $C$ the output state-mapping matrix of $\mathcal{G}$), whereas our analysis allows for an arbitrary positive-definite $Q$ and a $P$ matrix chosen as either $Q$ or the solution to the DARE.

\subsection{Matrix Symbol for the Prediction Matrix}

We start by analyzing the prediction matrix $\Gamma_{c}$ and note that its diagonals are constant blocks, which means that $\Gamma_{c}$ is a block Toeplitz matrix.
 Finite-dimensional block Toeplitz matrices with blocks of size $m \times n$ can be viewed as the truncation of an infinite-dimensional block Toeplitz matrix by extending the block pattern.
 The infinite-dimensional matrix is represented by a matrix-valued function mapping $\mathbb{T} \to \mathbb{C}^{m \times n}$, called the matrix symbol, which is used for analyzing properties of the matrix.

The blocks on the diagonals of the matrix give the spectral coefficients of the matrix symbol, so the symbol can be represented as a Fourier series with the coefficients given by the matrix blocks.
 For the original prediction matrix $\Gamma$, the truncated Fourier series that only uses the blocks in the horizon is given by
 \begin{equation*}
     \sum_{i=0}^{N-1} A^i B z^{-i}, \quad \forall z \in \mathbb{T}.
 \end{equation*}
 As the horizon length increases, this series is only guaranteed to converge to a finite matrix symbol when the system $\mathcal{G}$ with matrices $A$ and $B$ is Schur-stable.

To form a convergent Fourier series, we use the numerically robust CLQR formulation from Section~\ref{sec:numericallyRobust} to introduce a stabilizing linear state-feedback controller $u_{k} = -Kx_{k} + v_{k}$ to the prediction.
 For systems where the pair $(A,B)$ is stabilizable, this then leads to a convergent Fourier series for the prediction matrix $\Gamma_{c}$ of the new controlled system, and the finite matrix symbol given in Lemma~\ref{lem:gammaSymbolControlled}.

\begin{lemma}
    \label{lem:gammaSymbolControlled}
    Let the pair $(A,B)$ be stabilizable and $K \in \mathbb{R}^{m \times n}$ be a linear state-feedback control matrix used to form the prestabilized system~\eqref{eq:system:gk}.
    The prediction matrix $\Gamma_{c}$ then has the matrix symbol $\mathcal{P}_{\Gamma_{c}} \in \Cpi$ with
    \begin{equation*}
        \label{eq:pred:genFunc}
        \mathcal{P}_{\Gamma_{c}}(z) \coloneqq z(zI - (A-BK))^{-1}B = z \mathcal{G}_{c}(z),
        \quad
        \forall z \in \mathbb{T},
    \end{equation*}
    where $\mathcal{G}_c(\cdot)$ is the transfer function matrix for the system~$\mathcal{G}_c$.
\end{lemma}
\begin{proof}
    The matrix symbol for the block Toeplitz matrix $\Gamma_{c}$ is derived using the infinite-dimensional extension of $\Gamma_{c}$, called $\mathbf{\Gamma}$.
    The blocks that make up $\mathbf{\Gamma}$ can be extrapolated from \eqref{eq:mat:stableGamma} to
    \begin{equation}
    	\label{eq:gammaInfiniteBlocks}
    	\mathbf{\Gamma}_{i} \coloneqq \begin{cases}
    	0 & \text{if $i < 0$}, \\
    	(A-BK)^{i} B &  \text{if $i \geq 0$},
    	\end{cases}
    \end{equation}
    where $i$ is the number of the block diagonal of the matrix, with $0$ being the main diagonal and positive $i$ below the main diagonal.
    A possible way to form the matrix symbol of a block Toeplitz matrix is to define it as the trigonometric polynomial with the blocks of the matrix as coefficients \cite[\S 4.3]{Gutierrez-Gutierrez2012_blockSurvey}.
    Doing that for $\mathbf{\Gamma}$ uses the blocks~\eqref{eq:gammaInfiniteBlocks} as the coefficients to form the trigonometric polynomial
    \begin{equation}
        \label{eq:gammaPolySum}
    	\sum_{i=0}^{\infty} z^{-i} (A-BK)^{i} B.
    \end{equation}
    The constant $B$ matrix can be extracted from the summation, giving
    \begin{equation*}
  	    \left( \sum_{i=0}^{\infty} (A-BK)^i z^{-i} \right) B.
    \end{equation*}
	Since $K$ was designed to make $(A-BK)$ Schur-stable, the summation becomes a convergent Neumann series that converges to
    $z( zI - (A - BK) )^{-1}$ \cite[\S 3.4]{Peterson2012}.
    With the $B$ matrix right-multiplying the summation, the result is the transfer function matrix for the time-shifted system $z\mathcal{G}_{c}(z)$, giving the matrix symbol in the lemma.
    Finally, note that the coefficients in the sum~\eqref{eq:gammaPolySum} are absolutely summable, so $\mathcal{P}_{\Gamma_{c}}$ is in the Wiener class, meaning that $\mathcal{P}_{\Gamma_{c}} \in \Linf$ and is continuous and $2\pi$-periodic, leading to $\mathcal{P}_{\Gamma_{c}} \in \Cpi$.
\end{proof}

It is tempting to only apply the stabilizing controller $K$ after the horizon, like the CLQR stability theory in \cite{Mayne2000_StabilitySurvey}, but this will cause $\mathbf{\Gamma}$ to have $A^{i}B$ in the upper-left $N \times N$ block submatrix and $(A - BK)^{i}B$ in the remaining part, breaking the block Toeplitz structure.

At this point, there is no restriction on the type of linear state-feedback controller used to prestabilize the system --- any controller that results in a Schur-stable closed-loop system can be used.
 A convenient choice for $K$ is the infinite-horizon LQR controller designed using the cost matrices in \eqref{eq:mpc:lin:cost} with $P$ chosen to be the solution to the DARE~\eqref{eq:dare} as described in Section~\ref{sec:prelim:formulation}.

If the system $\mathcal{G}$ is Schur-stable to begin with, then there is no need for a prestabilizing controller $K$ and the matrices $\Gamma$ and $\Gamma_{c}$ can be the same. The matrix symbol for $\Gamma_{c}$ can then be simplified as shown in Corollary~\ref{cor:gammaSymbolSchurStable}.
\begin{corollary}
	\label{cor:gammaSymbolSchurStable}
	If the system $\mathcal{G}$ is Schur-stable, then with $K = 0$ the prediction matrix $\Gamma$ has a convergent Fourier series, producing the matrix symbol $\mathcal{P}_{\Gamma} \in \Cpi$ with
	\begin{equation*}
	\label{eq:pred:genFuncSchur}
	\mathcal{P}_{\Gamma}(z) \coloneqq z(zI - A)^{-1}B = z \mathcal{G}(z),
	\qquad
	\forall z \in \mathbb{T},
	\end{equation*}
	where $\mathcal{G}(\cdot)$ is the transfer function matrix for the system~$\mathcal{G}$.
\end{corollary}

\subsection{Matrix Symbol for the Hessian}
\label{sec:specRes:condHess}

The Hessian of the MPC problem formulation in~\eqref{eq:mpc:stablecondMPC} can be split into four distinct parts
 \begin{equation}
    \label{eq:dense:primalHessian:MatrixSplitting}
     H_{c} \coloneqq H_Q + H_P + H_{K} + H_R
 \end{equation}
 where $H_Q$, $H_R$, $H_P$, and $H_{K}$ are the parts that contain the matrices $Q$, $\bar{R}$, $P$, and the $\bar{K}$ cross-terms, respectively.
 Slightly different analysis must be done depending on the choice of $P$, and in this work we focus on the cases when $P=Q$ and $P$ is the solution to the DARE for the infinite-dimensional unconstrained LQR of problem.

\subsubsection{$P$ is the same as $Q$}

Choosing $P = Q$ for \eqref{eq:mpc:linMPC} allows the term $H_{P}$ to be consolidated into $H_{Q}$, giving
\begin{align*}
	H_{c} &= \trans{\Gamma_{c}} \bar{Q}_{c} \Gamma_{c} + \trans{\Gamma_{c}}\trans{\bar{K}}\bar{R} + \bar{R}\bar{K}\Gamma_{c} + \bar{R},\\
	\bar{Q}_{c} &= I_{N} \otimes (Q + \trans{K} R K).
\end{align*}
 Analysis of the resulting matrix $H_{c}$ reveals that $\mathbf{H}_{c}$ is also block Toeplitz with the matrix symbol given in Lemma~\ref{lem:HqSymbol}.
\begin{lemma}
    \label{lem:HqSymbol}
    Let $P=Q$ and $\mathcal{P}_{\Gamma_{c}}$ be the matrix symbol for $\Gamma_{c}$ from Lemma~\ref{lem:gammaSymbolControlled}. The matrix $\mathbf{H}_{c}$ is then a block Toeplitz matrix with the matrix symbol $\mathcal{P}_{H_{c}} \in \Cpi$, where
    \begin{multline*}
        \label{eq:dense:PQ:Hq:genFunc}
        \mathcal{P}_{H_{c}}(z) \coloneqq \ctrans{\mathcal{P}}_{\Gamma_{c}}(z) (Q + \trans{K} R K) \mathcal{P}_{\Gamma_{c}}(z) + R\\
        - \ctrans{\mathcal{P}}_{\Gamma_{c}}(z) \trans{K} R - R K \mathcal{P}_{\Gamma_{c}}(z), \quad
        \forall z \in \mathbb{T}.
    \end{multline*}
\end{lemma}
\begin{proof}
    Using the assumption that $P = Q$, we can see that the new state weighting matrix $\bar{Q}_{c}$ is block diagonal with the same entry in each block, making $\bar{Q}_{c}$ a block Toeplitz matrix with the symbol $\mathcal{P}_{\bar{Q}_{c}}(z) \coloneqq Q + \trans{K} R K$.
    Since $\mathbf{\Gamma}_{c}$ is a lower-triangular block matrix and $\trans{\mathbf{\Gamma}_{c}}$ is an upper-triangular block matrix, the product $\trans{\mathbf{\Gamma}_{c}} \bar{\mathbf{Q}} \mathbf{\Gamma}_{c}$ is block Toeplitz with matrix symbol
    $
    \ctrans{\mathcal{P}_{\Gamma_{c}}} \mathcal{P}_{\bar{Q}_{c}} \mathcal{P}_{\Gamma_{c}}
    $
    \cite[Lemma 4.5]{Gutierrez-Gutierrez2012_blockSurvey}.

    By construction, $\mathbf{\bar{R}}$ and $\mathbf{\bar{K}}$ are both block Toeplitz matrices with symbols $\mathcal{P}_{\bar{R}}(z) \coloneqq R$ and  $\mathcal{P}_{\bar{K}}(z) \coloneqq -K$, respectively.
    $\trans{\mathbf{\Gamma}_{c}} \trans{\mathbf{\bar{K}}} \mathbf{\bar{R}}$ and $\mathbf{\bar{R}} \mathbf{\bar{K}} \mathbf{\Gamma}_{c}$ are then block Toeplitz, since the product $\mathbf{\bar{R}} \mathbf{\bar{K}}$ produces a block diagonal block Toeplitz matrix, preserving the block Toeplitz structure of $\mathbf{\Gamma}_{c}$ during the multiplication.
    The matrix symbols for $\trans{\mathbf{\Gamma}_{c}} \trans{\mathbf{\bar{K}}} \mathbf{\bar{R}}$ and $\mathbf{\bar{R}} \mathbf{\bar{K}} \mathbf{\Gamma}_{c}$ are then $- \ctrans{\mathcal{P}}_{\Gamma_{c}}(z) \trans{K} R$ and $- R K \mathcal{P}_{\Gamma_{c}}(z)$, respectively.
    Block Toeplitz structure is preserved over the addition of two or more block Toeplitz matrices, and the matrix symbol for the resulting sum is the sum of the original matrix symbols, giving the matrix symbol in the lemma.
\end{proof}

It is important to note that the product of block Toeplitz matrices is not guaranteed to be block Toeplitz except in certain special cases, while the addition of multiple block Toeplitz matrices with compatible block sizes is always guaranteed to produce a block Toeplitz result.
 In this case, the lower-triangular structure of $\mathbf{\Gamma}_{c}$ and the block Toeplitz structure of $\bar{\mathbf{Q}}_{c}$ implies that the product $\trans{\mathbf{\Gamma}_{c}} \bar{\mathbf{Q}} \mathbf{\Gamma}_{c}$ is one of the special cases where the multiplication of the three block Toeplitz matrices of compatible block sizes is block Toeplitz.
 Choosing an arbitrary $P$ with $P \neq Q$ will cause $\bar{Q}_{c}$ to no longer be block Toeplitz, so the multiplication will not necessarily produce a block Toeplitz matrix and the Hessian may not be block Toeplitz.

\subsubsection{$P$ is the solution to the DARE}

\begin{figure*}[t!]
	\begin{equation}
	\label{mat:HQ}
	H_{Q} \coloneqq
	\begin{bmatrix}
	\sum_{i=0}^{N-1} \trans{B} \trans{(A_c^{i})} Q_c A_c^{i} B   & \sum_{i=0}^{N-2} \trans{B} \trans{(A_c^{i+1})} Q_c A_c^{i} B & \sum_{i=0}^{N-3} \trans{B} \trans{(A_c^{i+2})} Q_c A_c^{i} B & \cdots & \trans{B} \trans{(A_c^{N-1})} Q_c B\\
	\sum_{i=0}^{N-2} \trans{B} \trans{(A_c^{i})} Q_c A_c^{i+1} B & \sum_{i=0}^{N-2} \trans{B} \trans{(A_c^{i})} Q_c A_c^{i} B   & \sum_{i=0}^{N-3} \trans{B} \trans{(A_c^{i+1})} Q_c A_c^{i} B & \cdots & \trans{B} \trans{(A_c^{N-2})} Q_c B\\
	\sum_{i=0}^{N-3} \trans{B} \trans{(A_c^{i})} Q_c A_c^{i+2} B & \sum_{i=0}^{N-3} \trans{B} \trans{(A_c^{i})} Q_c A_c^{i+1} B & \sum_{i=0}^{N-3} \trans{B} \trans{(A_c^{i})} Q_c A_c^{i} B   & \cdots & \trans{B} \trans{(A_c^{N-3})} Q_c B\\
	\vdots & \vdots & \vdots & \ddots & \vdots \\
	\trans{B} Q_c A_c^{N-1} B & \trans{B} Q_c A_c^{N-2} B & \trans{B} Q_c A_c^{N-3} B & \cdots & \trans{B} Q_c B
	\end{bmatrix}
	\end{equation}
	\begin{equation}
	\label{mat:HP}
	H_{P} \coloneqq
	\begin{bmatrix}
	\trans{B} \trans{(A_c^{N})} P A_c^{N} B   & \trans{B} \trans{(A_c^{N})} P A_c^{N-1} B   & \trans{B} \trans{(A_c^{N})} P A_c^{N-2} B   & \cdots & \trans{B} \trans{(A_c^{N})} P A_c B\\
	\trans{B} \trans{(A_c^{N-1})} P A_c^{N} B & \trans{B} \trans{(A_c^{N-1})} P A_c^{N-1} B & \trans{B} \trans{(A_c^{N-1})} P A_c^{N-2} B & \cdots & \trans{B} \trans{(A_c^{N-1})} P A_c B\\
	\trans{B} \trans{(A_c^{N-2})} P A_c^{N} B & \trans{B} \trans{(A_c^{N-2})} P A_c^{N-1} B & \trans{B} \trans{(A_c^{N-2})} P A_c^{N-2} B & \cdots & \trans{B} \trans{(A_c^{N-2})} P A_c B\\
	\vdots                                  & \vdots                                    & \vdots                          & \ddots & \vdots \\
	\trans{B} \trans{A_c} P A_c^{N} B       & \trans{B} \trans{A_c} P A_c^{N-1} B       & \trans{B} \trans{A_c} P A_c^{N-2} B       & \cdots & \trans{B} \trans{A_c} P A_c B
	\end{bmatrix}
	\end{equation}
	\hrulefill
\end{figure*}

Choosing $P$ as the solution to the DARE~\eqref{eq:dare} causes $\bar{Q}_{c}$ to not be block Toeplitz, since $\bar{Q}_{c}$ will then have a different matrix in the lower-right corner than the rest of the main diagonal.
 This means that the analysis based on the multiplication of structured block Toeplitz matrices used in the proof of Lemma~\ref{lem:HqSymbol} no longer can be applied.
 However, the resulting $H_{c}$ matrix is still block Toeplitz due to the fact that $P$ will represent the cost of the controller applied after the horizon ends.

\begin{proposition}
	\label{prop:Hpsymbol}
	If $P$ is the solution to the DARE~\eqref{eq:dare} and $K$ is the infinite-horizon LQR controller for $\mathcal{G}$, then $H_{c}$ is block Toeplitz and has the same matrix symbol as the case when $P=Q$ given in Lemma~\ref{lem:HqSymbol}.
\end{proposition}
\begin{proof}
	Using the matrix splitting \eqref{eq:dense:primalHessian:MatrixSplitting}, the Hessian can be decomposed into four terms, with $H_{Q}$ and $H_{P}$ given by \eqref{mat:HQ} and \eqref{mat:HP}, respectively, and $H_{R}$ and $H_{K}$ the same as when $P=Q$.
	We start by examining the first diagonal term of $H_{Q} + H_{P}$,
	\begin{equation}
		\label{eq:proof:HpDiagonal}
		\sum_{i=0}^{N-1} \trans{B} \trans{(A_c^{i})} Q_c A_c^{i} B + \trans{B} \trans{(A_c^{N})} P A_c^{N} B.
	\end{equation}
	Since $P$ is the solution to the DARE~\eqref{eq:dare}, $P$ is also the solution to the Lyapunov equation
	\begin{equation}
	    \label{eq:proof:controlDlayp}
	    \trans{(A - BK)} P (A - BK) + Q + \trans{K} R K = P
	\end{equation}
	when $K$ is the infinite-horizon LQR controller.
	This means that $P$ can also be expressed as
	$$ P = \sum_{i = 0}^{\infty} (\trans{(A - BK)})^{i}(Q + \trans{K}RK)(A - BK)^{i},$$
	transforming~\eqref{eq:proof:HpDiagonal} into
	\begin{equation*}
		\sum_{i=0}^{N-1} \trans{B} \trans{(A_c^{i})} Q_c A_c^{i} B + \trans{B} \trans{(A_c^{N})} \left( \sum_{i = 0}^{\infty} \trans{(A_{c}^{i})} Q_{c} A_{c}^{i} \right) A_c^{N} B.
	\end{equation*}
	The $\trans{(A_{c}^{N})}$ and $A_{c}^{N}$ terms around the right summation can be consolidated into the summation, offsetting its starting point to be $i = N$ instead of $0$ --- meaning the right summation is simply the continuation of the left summation to infinity and allowing the two sums to be consolidated into
	\begin{equation}
		\label{eq:proof:Hpdiagonalelem}
		\sum_{i=0}^{\infty} \trans{B} \trans{(A_c^{i})} Q_c A_c^{i} B.
	\end{equation}

	The same analysis can be performed on the other terms on the main diagonal, which only differ by where the left summation ends and the right summation is offset to.
	Therefore, the main diagonal of the matrix sum $H_{Q} + H_{P}$ is composed of blocks with all the same terms.
	A similar analysis can be done on all diagonals above and below the main diagonal, showing that they are also composed of blocks with all the same terms down the diagonal.

	Since all the diagonals are composed of the same blocks down their length, the matrix sum $H_{Q} + H_{P}$ is block Toeplitz, and the resulting Hessian $H_{c}$ is block Toeplitz as well, since $H_{R}$ and $H_{K}$ are already known to be block Toeplitz.

	To construct the matrix symbol for $H_{c}$ when $P$ is the solution to the DARE, we examine the elements in the matrix $H_{Q} + H_{P}$ and how they relate to the case when $P = Q$.
	Note that when $P = Q$, the individual elements of the matrix $H_{Q}$ have a summation that terminates at the horizon length.
	Since the matrix symbol is based on the infinite-dimensional matrix, if the matrix $H_{Q}$ is extrapolated to a horizon of infinity to form $\mathbf{H}_{Q}$, the summations in $H_{Q}$ will all terminate at infinity.
	Therefore, the sum $H_{Q} + H_{P}$ will have the same blocks as the infinite-dimensional $\mathbf{H}_{Q}$ when $P = Q$, so the Hessians for the cases when $P = Q$ and $P$ is the solution to the DARE will both have the same matrix symbol.
\end{proof}

\subsubsection{Simplification when $\mathcal{G}$ is Schur-stable}

When $\mathcal{G}$ is Schur-stable and the results in Corollary~\ref{cor:gammaSymbolSchurStable} are used to simplify the matrix symbol of the prediction matrix, the results given in Lemma~\ref{lem:HqSymbol} and Proposition~\ref{prop:Hpsymbol} can be simplified as well.

\begin{corollary}
	\label{cor:HSymbolSchurStable}
	If the system $\mathcal{G}$ is Schur-stable, then with $K = 0$ and $P = Q$ or $P$ the solution to the discrete Lyapunov equation \eqref{eq:dlyap}, the Hessian $H$ has the matrix symbol $\mathcal{P}_{H} \in \Cpi$ with
	\begin{equation*}
	\label{eq:H:genFuncSchur}
        \mathcal{P}_{H}(z) \coloneqq \ctrans{\mathcal{P}}_{\Gamma}(z) Q  \mathcal{P}_{\Gamma}(z) + R, \quad
\forall z \in \mathbb{T}.
	\end{equation*}
	where $\mathcal{G}(\cdot)$ is the transfer function matrix for the system~$\mathcal{G}$.
\end{corollary}

Note that the matrix symbol in Corollary~\ref{cor:HSymbolSchurStable} has the same form as the matrix symbol in Lemma~\ref{lem:HqSymbol}, just with the terms containing $K$ removed and a different choice of the matrix $P$.

In order for Proposition~\ref{prop:Hpsymbol} to reduce to Corollary~\ref{cor:HSymbolSchurStable} in the Schur-stable case with $K=0$, the terminal cost must be based on the solution to the discrete Lyapunov equation instead of the DARE, since using the DARE solution with no prestabilizing controller will cause the Lyapunov equation~\eqref{eq:proof:controlDlayp} used in the proof of Proposition~\ref{prop:Hpsymbol} to not be valid.

\subsection{Spectral Bounds for the Hessian}
\label{sec:specRes:bounds}

One useful property of block Toeplitz matrices is that the eigenvalue spectrum for any finite-dimensional truncation of the infinite-dimensional block Toeplitz matrix is contained within the extremal eigenvalues of its matrix symbol.
 This means that the minimum and maximum eigenvalues of the matrix $H_{c}$ can be bounded by analyzing the matrix symbol $\mathcal{P}_{H_{c}}$, since $H_{c}$ is a finite-dimensional truncation of the infinite-dimensional matrix $\mathbf{H_{c}}$ to $N$ blocks.

\begin{theorem}
    \label{thm:dense:hessEig}
    Let $H_{c}$ be the condensed Hessian with $P$ the solution to the DARE and a prediction horizon of length $N$ that is block Toeplitz with matrix symbol $\mathcal{P}_{H_{c}}$ given in Lemma~\ref{lem:HqSymbol}, then the following hold:
    \begin{enumerate}[label={(\alph*)},ref={\thetheorem(\alph*)}]
        \item $ \lambda_{min}( \mathcal{P}_{H_{c}} ) \leq  \lambda(H_{c}) \leq \lambda_{max}( \mathcal{P}_{H_{c}}) $
        \label{thm:dense:hessEig:bounds}
        \item $ \underset{N \to \infty}{\lim} \kappa(H_{c}) = \kappa( \mathcal{P}_{H_{c}} ) $
        \label{thm:dense:hessEig:condLimit}
    \end{enumerate}
\end{theorem}
\begin{proof}\hfill~
    \begin{enumerate}[label={(\alph*)},ref={\thetheorem(\alph*)}]
        \item The spectrum of a finite-dimensional truncation of a block Toeplitz matrix with a symbol in $\Cpi$ is bounded by the extremes of the spectrum of its symbol \cite[Theorem 4.4]{Gutierrez-Gutierrez2012_blockSurvey}.
        \item Note that $H_{c}$ is a Hermitian matrix, which means that it is also normal \cite[\S 4.1]{Horn2013}. Since it is both normal and positive semi-definite, the singular values are the same as the eigenvalues \cite[\S 3.1]{Horn1994}, resulting in the condition number becoming $\kappa(H_{c}) = \frac{\lambda_{n}(H_{c})}{\lambda_{1}(H_{c})}$. Taking the limit of both sides in conjunction with the spectral bounds from part~(a) gives
        \begin{equation*}
        \underset{N \to \infty}{\lim} \kappa(H_{c})
        =
        \kappa( \mathcal{P}_{H_{c}} ).
        \end{equation*}
    \end{enumerate}
\end{proof}

\begin{corollary}
    \label{cor:stableHessEig}
	If the system $\mathcal{G}$ is Schur-stable, then with $K = 0$ and $P$ the solution to the discrete Lyapunov equation~\eqref{eq:dlyap}, the results in Theorem~\ref{thm:dense:hessEig} can be applied to $H$ using the matrix symbol $\mathcal{P}_{H}$ instead of $\mathcal{P}_{H_{c}}$.
\end{corollary}

Essentially, these results say that the spectrum for the condensed Hessian in these cases will always be contained inside the interval defined by the maximum and minimum eigenvalues of the matrix symbol in Lemma~\ref{lem:HqSymbol}/Corollary~\ref{cor:HSymbolSchurStable}.
 Additionally, as $N \to \infty$ the extremal eigenvalues of the Hessian will converge asymptotically to the maximum and minimum eigenvalues of its symbol.

\section{Preconditioning}
\label{sec:precond}

The spectral results presented in Sections~\ref{sec:specRes:condHess} and~\ref{sec:specRes:bounds} can be readily extended to analyze the case of a preconditioned Hessian, as well as to help design new preconditioners.

\subsection{Analysis of the Preconditioned Hessian}
\label{sec:precond:analysis}

For simplicity of discussion, we focus on the case when $H_{c}$ is symmetrically preconditioned to form
\begin{equation}
    \label{eq:precond:symmetric}
    H_{L} \coloneqq L_{N}^{-1} H_{c} \trans{(L_{N}^{-1})},
\end{equation}
 where $L_{N}$ is a block-diagonal preconditioner that has $N$ copies of the matrix $L$ on its diagonal, thus guaranteeing that the preconditioned matrix is block Toeplitz.
 This case is most appropriate for first-order methods, since it guarantees that the structure of the feasible set is preserved over the preconditioning operation and that the preconditioned Hessian is symmetric \cite{Richter2012_FGMcomplexity}.

Since the preconditioner matrix $L_{N}$ is block-diagonal with only~$L$ on its main diagonal, the matrix symbol for $L_{N}$ is simply~$L$.
 The spectral bounds in Section~\ref{sec:specRes:bounds} can then be extended to the preconditioned matrix $H_{L}$ by simply replacing $\mathcal{P}_{H_{c}}$ in Theorem~\ref{thm:dense:hessEig:condLimit} with $\mathcal{P}_{H_{L}}$ given by
 \begin{equation}
     \label{eq:precond:toepSymbol}
     \mathcal{P}_{H_{L}} \coloneqq \bar{L} \mathcal{P}_{H_{c}} \trans{\bar{L}},
 \end{equation}
 where $\bar{L} \coloneqq L^{-1}$.
 A similar substitution can be made when $\mathcal{G}$ is Schur-stable by using $\mathcal{P}_{H}$ in \eqref{eq:precond:toepSymbol} with Corollary~\ref{cor:stableHessEig}.
 This analysis requires $L_{N}$ to be block Toeplitz, which may not be the case for many methods of computing the preconditioner, unless such a constraint is added to its computation.

\subsection{Preconditioner Design}
\label{sec:precond:design}

There is a rich literature of preconditioners for Toeplitz and circulant matrices, with a focus on designing the preconditioners independent of the size of the matrix (see \cite{Chan2007} and references therein).
 These existing ideas can be applied to the block Toeplitz structure of the Hessian in the CLQR problem in order to design preconditioners to use when solving the QP.

One of the first circulant preconditioners was proposed by Gilbert Strang in \cite{strangProposalToeplitzMatrix1986}.
 This preconditioner was originally proposed for preconditioning iterative conjugate gradient methods, and is formed by simply copying the central diagonals of the Toeplitz matrix into the preconditioning matrix and wrapping them around to form a circulant matrix.
 Strang's preconditioner can be naturally extended to the block Toeplitz case by simply copying the individual blocks into the preconditioning matrix and wrapping them around to form a block circulant matrix.

In the case of the CLQR problem with a block diagonal preconditioning matrix, we can explicitly compute the block on the diagonal of the preconditioning matrix without forming the entire Hessian, as shown in Theorem~\ref{thm:precond:circ}.
 \begin{theorem}
    \label{thm:precond:circ}
    Let $H_{c}$ be the Hessian from \eqref{eq:mpc:stablecondMPC} formed by choosing either:
    \begin{itemize}
        \item $K$ as the infinite-horizon LQR controller for $\mathcal{G}$, with $P$ the solution to the DARE~\eqref{eq:dare}, or
        \item $K = 0$ with $P$ the solution to the discrete-time Lyapunov equation~\eqref{eq:dlyap} for a Schur-stable $\mathcal{G}$.
    \end{itemize}
    The matrix $H_{c}$ can be symmetrically preconditioned as $L_{N}^{-1} H_{c} \trans{(L_{N}^{-1})}$, with $L_{N} = I_{N} \otimes L$ and the blocks $L$ given by the lower-triangular Cholesky decomposition of $M$ with
    \begin{equation*}
        M \coloneqq \trans{B} P B - \trans{B} \trans{K} R - R K B + R.
    \end{equation*}
 \end{theorem}
 \begin{proof}
 Based on the work in \cite{strangProposalToeplitzMatrix1986}, a Circulant preconditioning matrix $W$ for the block Toeplitz matrix $V$ will have entries that are obtained by copying the central diagonals of $V$ and wrapping them around to form a circulant matrix  since the main diagonals are usually strongly dominant.
 For example, the matrix
 \begin{equation*}
    V \coloneqq \begin{bmatrix}
    V_{0} & V_{1} & V_{2} & V_{3} & V_{4} \\
    V_{-1} & V_{0} & V_{1} & V_{2} & V_{3} \\
    V_{-2} & V_{-1} & V_{0} & V_{1} & V_{2} \\
    V_{-3} & V_{-2} & V_{-1} & V_{0} & V_{1} \\
    V_{-4} & V_{-3} & V_{-2} & V_{-1} & V_{0} \\
    \end{bmatrix},
 \end{equation*}
 will have a block circulant preconditioning matrix
 \begin{equation*}
    W \coloneqq \begin{bmatrix}
    V_{0}  & V_{1}  & V_{2}  & V_{-2} & V_{-1} \\
    V_{-1} & V_{0}  & V_{1}  & V_{2}  & V_{-2} \\
    V_{-2} & V_{-1} & V_{0}  & V_{1}  & V_{2} \\
    V_{2}  & V_{-2} & V_{-1} & V_{0}  & V_{1} \\
    V_{1}  & V_{2}  & V_{-2} & V_{-1} & V_{0} \\
    \end{bmatrix}.
 \end{equation*}
 Since we want the preconditioner to be block diagonal, we only need to focus on computing the diagonal block $V_0$ for the CQLR.

 For the block Toeplitz Hessian $H_{c}$, the main diagonal block of the infinite dimensional block Toeplitz matrix is \eqref{eq:proof:Hpdiagonalelem}, which when the $H_{R}$ and $H_{K}$ components are added becomes
 \begin{equation*}
 	 \sum_{i=0}^{\infty} \trans{B} \trans{(A_{c}^{i})} Q_{c} A_{c}^{i} B - \trans{B} \trans{K} R - R K B + R.
 \end{equation*}
 Since $A_{c}$ is Schur-stable and $K$ is the LQR controller, this sum converges to the solution of the DARE, making the diagonal block
 \begin{equation*}
    V_{0} = \trans{B} P B - \trans{B} \trans{K} R - R K B + R.
 \end{equation*}
 Since Strang's preconditioner is designed as a left preconditioner (i.e.\  $\invnb{W}V$) and we want a symmetric preconditioner, we apply the lower-triangular Cholesky factorization to $M$, resulting in $L$.
 \end{proof}

\begin{remark}
	The preconditioning matrix $L$ in Theorem~\ref{thm:precond:circ} can also be formed by using the matrix square-root of $M$ instead of the Cholesky factorization, which could allow for $L$ to have the same structure as $M$ if the square root operation is structure preserving (e.g.\ if $\mathcal{G}$ is a circulant system, using the matrix square root can lead to $L$ and $H_{L}$ being circulant as well).
\end{remark}

Note that the matrices $M$ and $L$ will have dimension $m \times m$, and that the full preconditioning matrix $L_{N}$ is formed by simply repeating $L$ down the diagonal $N$ times, so changing the horizon length means simply adding or removing blocks of $L$ from the diagonal of $L_{N}$.
 
Since $L_{N}$ is block Toeplitz, we can show that if $\mathcal{G}$ is single-input, the preconditioner will not affect the conditioning.
\begin{proposition}
    \label{prop:cond:singleInput}
    If the system $\mathcal{G}$ is single-input (i.e.\ $m = 1$), then the block Toeplitz preconditioner $L_{N}$ will not affect the condition number of $H_{c}$ (i.e.\ $\kappa(H_{L}) = \kappa(H_{c})$).
\end{proposition}
\begin{proof}
    A system with $m = 1$ will have $L \in \mathbb{R}$, making the preconditioner matrix simply $L_{N} = I_{N} \otimes L = L I_{N}$.
    For the symmetric case, \eqref{eq:precond:symmetric} becomes $H_{L} = \left(\frac{1}{L} I_{N}\right) H_{c} \left(I_{N} \frac{1}{L}\right)$, which simplifies to $H_{L} = \frac{1}{L^2} H_{c}$.
    Using the fact that $\lambda_i(\alpha H_c) = \abs{\alpha}\lambda_i(H_{c})$,
    \begin{equation*}
        \kappa(H_{L})
        = \frac{\lambda_{max}\left( \frac{1}{L^2}H_{c} \right)}{\lambda_{min}\left(\frac{1}{L^2} H_{c}\right)}
        = \frac{\abs{L^2}}{\abs{L^2}} \frac{\lambda_{max}(H_{c})}{\lambda_{min}(H_{c})}
        = \kappa(H_{c}).
    \end{equation*}
\end{proof}

Also note that all the matrices involved in computing $M$ have dimensions on the order of $m$ and $n$, and have no relation to the horizon length.
 This is in contrast to SDP-based preconditioner design techniques such as \cite{Richter2012_FGMcomplexity}, which require the full Hessian to be placed inside the semidefinite optimization problem.

\section{Numerical Experiments}
\label{sec:comparisons}

In this section, we present numerical examples showing the spectral properties computed using the results of Section~\ref{sec:spectralProperties}, and also the effect of applying the preconditioner from Section~\ref{sec:precond:design} to the CLQR problem for four systems.

\subsection{Example Systems}

\subsubsection{Schur-stable system}
The first two example systems we use are the Schur-stable discrete-time system with four states and two inputs given in \cite{Jones2008} with state equation
 \begin{align*}
    x_{k+1} = \begin{bmatrix}
    0.7 & -0.1 & 0.0 & 0.0 \\
    0.2 & -0.5 & 0.1 & 0.0 \\
    0.0 &  0.1 & 0.1 & 0.0 \\
    0.5 &  0.0 & 0.5 & 0.5
    \end{bmatrix}
    x_{k} +
    \begin{bmatrix}
    0.0 & 0.1 \\
    0.1 & 1.0 \\
    0.1 & 0.0 \\
    0.0 & 0.0
    \end{bmatrix}
    u_{k}.
 \end{align*}
 We constrain the inputs of the system to be $\abs{u_{k}} \leq 0.5$ and use a prediction horizon of $N=10$.
 The first system has cost matrices
 \begin{equation}
    \label{eq:schurEx:normal}
    Q = \text{diag}(10, 20, 30, 40), \quad
    R = \text{diag}(10, 20),
 \end{equation}
 and the second system has cost matrices
 \begin{equation}
    \label{eq:schurEx:illcond}
    Q = \text{diag}(100, 200, 300, 400), \quad
    R = \text{diag}(0.001, 0.002).
 \end{equation}

\subsubsection{Inverted pendulum}

The next example system we use is a linearized inverted pendulum described by the continuous-time dynamics
 \begin{align*}
    \dot{x} = \begin{bmatrix}
                 0 &  1 & 0 & 0 \\
    \frac{3 g}{2l} & -b & 0 & 0 \\
                 0 &  0 & 0 & 1 \\
                 0 &  0 & 0 & 0
    \end{bmatrix}
    x +
    \begin{bmatrix}
    0 \\
    \frac{3}{2l} \\
    0 \\
    1
    \end{bmatrix}
    u,
 \end{align*}
 with $g = 9.8067$, $b = 1$, and $l = 0.21$.
 The system was discretized using a zero-order hold with a sampling time of 0.02\,s, resulting in an unstable discrete-time system.
 The CLQR problem used the cost matrices
 $ Q = \text{diag}(1000, 1, 100, 1), R = 10, $
 a prediction horizon of $N=10$, and input constraints $\abs{u} \leq 10$,

\subsubsection{Distillation column}

The final example system we use is a binary distillation column with 11 states and 3 inputs from \cite[Problem 90-01]{davisonBenchmarkProblemsControl1990}.
 The system was discretized using a zero-order hold with a sampling time of 1.0\,s, resulting in a Schur-stable system.
 We used a prediction horizon of $N=100$ and the cost matrices
 \begin{align*}
    Q &= \text{diag}(10, 20, 30, 40, 50, 60, 70, 80, 90, 100, 110), \\
    R &= \text{diag}(10, 20, 30),
 \end{align*}
 with input constraints
 $\abs{u_1}~{\leq}~2.5,$
 $\abs{u_2}~{\leq}~2.5,$
 $\abs{u_3}~{\leq}~0.30.$

\subsection{Matrix Conditioning}

\input{figures/spectrum_schur}

The first numerical results we present examine the Hessian conditioning.
 The preconditioners were implemented using Julia 1.6.1, and the COSMO solver (version 0.8.1) \cite{garstkaCOSMOConicOperator2019} was used to compute the SDP preconditioner from \cite{Richter2012_FGMcomplexity}.
 Figures~\ref{fig:schur:original} and~\ref{fig:pendulum:original} show that the limits presented in Theorem~\ref{thm:dense:hessEig:condLimit} are attained as the horizon length increases.
 Note that the rate at which the finite-dimensional Hessian's condition number approaches the bound is system-dependent, with the Schur-stable system~\eqref{eq:schurEx:normal} converging within 0.01\% of the bound by $N=40$, and the inverted pendulum system by $N = 225$.

As described in Section~\ref{sec:precond:analysis}, the condition number bound of Theorem~\ref{thm:dense:hessEig:condLimit} can be used to analyze the preconditioned Hessian, and is shown in Figures~\ref{fig:schur:pedcond} and~\ref{fig:pendulum:pedcond}.
 The bound was computed for the proposed preconditioner by finding $L$ using Theorem~\ref{thm:precond:circ}, then substituting \eqref{eq:precond:toepSymbol} into Theorem~\ref{thm:dense:hessEig:condLimit}.
 Note that the bound computed with $L$ from Theorem~\ref{thm:precond:circ} does not hold for the SDP preconditioner from \cite{Richter2012_FGMcomplexity}, since $L$ will be different between the two preconditioners.
 Additionally, the SDP preconditioner does not guarantee that $L_{N}$ will be block Toeplitz, so Theorem~\ref{thm:dense:hessEig} cannot be used to compute horizon-independent limits when it is used.

Comparing the proposed preconditioner against the existing SDP preconditioner shows that they produce nearly identical condition numbers, as can be seen in Figures~\ref{fig:schur:pedcond} and~\ref{fig:pendulum:pedcond} and Table~\ref{tab:cond}.
 Proposition~\ref{prop:cond:singleInput} can also be seen in Figure~\ref{fig:pendulum:pedcond}, where the conditioning of the Hessian is not affected by the preconditioners, since the inverted pendulum is a single-input system.

As shown in Table~\ref{tab:cond}, both the SDP preconditioner and our proposed preconditioner decrease the condition number by 66\% for the Schur-stable system~\eqref{eq:schurEx:normal} and the already Schur-stable non-prestabilized distillation column, and decrease the condition number for the ill-conditioned Schur-stable system~\eqref{eq:schurEx:illcond} by 97\%.
 Applying the preconditioner to the prestabilized inverted pendulum system has no effect, leaving the condition number nearly identical to the original, since this system is single-input.

\input{figures/spectrum_invertedPendulum}

\subsection{Performance}

To test the effect of the preconditioners on the performance of a first-order method for the CLQR problem, the Fast Gradient Method (FGM) was implemented in Julia using the constant step-size scheme \cite{Jerez2014} and gradient map stopping criteria \cite[Section 6.3.1]{Richter2012} with a desired error of $10^{-5}$.
 The inequality constraints were implemented through projection operations, with the non-prestabilized problems being simple projections onto a box and the prestabilized problems requiring a more complex projection algorithm.
 In these examples, the projections for the prestabilized problems were computed by solving the projection QP directly; however, in an embedded application other techniques like an explicit QP \cite{Merkli2015} can be used instead.

We present two types of iteration results in Table~\ref{tab:iter}: the actual iterations taken by the FGM and the cold-start Upper Iteration Bound (UIB) from \cite{Richter2012_FGMcomplexity}.
 The UIB will be the worst-case number of iterations needed to achieve convergence regardless of the number of active constraints in the solution, while the actual iterations is the number taken on a single example QP for the CLQR problem.
 The proposed preconditioner gives a 2.1x actual and 2.6x UIB speedup for the Schur-stable system~\eqref{eq:schurEx:normal}, a 3.6x actual and 2.25x UIB speedup for the non-prestabilized distillation column, and a 4.5x actual and 9.5x UIB speedup for the ill-conditioned Schur-stable system~\eqref{eq:schurEx:illcond}.

When the SDP preconditioner is applied to the non-prestabilized inverted pendulum, it has no effect on the iterations required for the FGM to converge, requiring 51 iterations in both cases.
 Applying a prestabilizing controller to the inverted pendulum gives a 12.75x speedup, while prestabilizing the distillation column produces a speedup of 9.6x, a larger speedup than just preconditioning the non-prestabilized system.

The proposed preconditioner is also faster to compute than the SDP preconditioner, with timing results for the example systems given in Table~\ref{tab:compTimes}.
 These results show at least an order of magnitude difference in the computation times, with the proposed preconditioner requiring 2.5\,ms to compute for the non-prestabilized distillation column versus 151.74 seconds for the SDP preconditioner.

\begin{table}[t!]
    \center
    \caption{Condition number of the preconditioned Hessian.}
    \label{tab:cond}
    \begin{tabular}{cccc}
        \textbf{System} & \textbf{Original} & \textbf{SDP \cite{Richter2012_FGMcomplexity}} & \textbf{Proposed (Thm.~\ref{thm:precond:circ})}  \\
        \hline Schur-stable~\eqref{eq:schurEx:normal} & 8.776 & 2.922 & 2.933 \\
        \hline \makecell{Ill-conditioned \\ Schur-stable~\eqref{eq:schurEx:illcond}} & 254.66 & 7.415 & 7.500  \\
        \hline \makecell{Inverted pendulum \\ (non-prestabilized)} & 42.512 & 42.468 & (Not computable) \\
        \hline \makecell{Inverted pendulum \\ (LQR prestabilized)} & 1.889 & 1.884 & 1.889 \\
        \hline \makecell{Distillation column \\ (non-prestabilized)} & 21.527 & 7.175 & 7.175\\
        \hline \makecell{Distillation column \\ (LQR prestabilized)} & 3.004 & 1.017 & 1.025\\
    \end{tabular}
\end{table}
\begin{table}[t!]
    \center
    \begin{threeparttable}
        \caption{Iterations required for cold-start convergence of the Fast Gradient Method to $\epsilon {=} 10^{-5}$.}
        \label{tab:iter}
        \begin{tabular}{cScScSc}
            \textbf{System} & \textbf{None}\tnote{1} & \textbf{SDP \cite{Richter2012_FGMcomplexity}\tnote{1}} & \textbf{Proposed (Thm.~\ref{thm:precond:circ})}\tnote{1}  \\
            \hline Schur-stable~\eqref{eq:schurEx:normal} & 42 / 19 & 16 / 9 & 16 / 9 \\
            \hline \makecell{Ill-conditioned \\ Schur-stable~\eqref{eq:schurEx:illcond}} & 294 / 114 & 32 / 25 & 31 / 25  \\
            \hline \makecell{Inverted pendulum \\ (non-prestabilized)} & 143 / 51 & 129 / 51 & (Not computable) \\
            \hline \makecell{Inverted pendulum \\ (LQR prestabilized)} & 18 / 3 & 17 / 3 & 18 / 3 \\
            \hline \makecell{Distillation column \\ (non-prestabilized)} & 97 / 48 & 43 / 25 & 43 / 25\\
            \hline \makecell{Distillation column \\ (LQR prestabilized)} & 22 / 5 & 3 / 3 & 3 / 3 \\
        \end{tabular}
        \begin{tablenotes}
        \item [1] Cold-start upper iteration bound from \cite{Richter2012_FGMcomplexity} / actual iterations
        \end{tablenotes}
    \end{threeparttable}
\end{table}
\begin{table}[t!]
    \center
    \caption{Time required for computing the preconditioners.}
    \label{tab:compTimes}
    \begin{tabular}{cccc}
        \textbf{System} & \textbf{SDP \cite{Richter2012_FGMcomplexity}} (ms) & \textbf{Proposed (Thm.~\ref{thm:precond:circ})} (ms) \\
        \hline Schur-stable~\eqref{eq:schurEx:normal} & 197.4 & 0.213 \\
        \hline \makecell{Ill-conditioned \\ Schur-stable~\eqref{eq:schurEx:illcond}} & 142.9 & 0.218  \\
        \hline \makecell{Inverted pendulum \\ (non-prestabilized)} & 18.69 & (Not computable) \\
        \hline \makecell{Inverted pendulum \\ (LQR prestabilized)} & 17.45  & 0.218  \\
        \hline \makecell{Distillation column \\ (non-prestabilized)} & 151746 & 2.543 \\
        \hline \makecell{Distillation column \\ (LQR prestabilized)} & 81929 & 2.545 \\
    \end{tabular}
\end{table}

\section{Conclusions}

In this work we presented a new preconditioner that is based on the block Toeplitz structure of the Hessian for the condensed CLQR problem when using either a prestabilizing LQR controller or a Schur-stable system with appropriate choice of $P$.
 We showed that this preconditioner is comparable to an existing SDP preconditioner on several numerical examples, and can lead to speedups of between 2.1x and 9.6x for the Fast Gradient Method.
 The proposed preconditioner is also faster to compute than the SDP preconditioner, especially for long-horizon problems like the distillation column example.

This work also highlighted the relationship between the transfer function and the spectrum of the condensed Hessian.
 We derived results relating the extrema of the condensed Hessian's spectrum to the extrema of the spectrum for a complex-valued matrix symbol formed using the weighting matrices and the system's transfer function, and showed that these results can be used to obtain bounds on the condition number of the Hessian when using our proposed preconditioner.
 The examples also showed that the numerical prestabilization controller $K$ has a direct effect on the spectrum of the Hessian, so preconditioning could also be achieved through a careful choice of $K$ instead of applying a separate preconditioner (at the expense of turning the constraint sets into more complex shapes).
 Future work could explore developing a preconditioner based on loop-shaping of the system to reduce the Hessian's condition number.

The derivation and examples in this work focused on preconditioning the primal QP for the CLQR problem, however it is also common to use a dual form of \eqref{eq:mpc:condMPC} with gradient algorithms such as the Dual Gradient Projection or Dual Fast Gradient Method.
 The preconditioner defined in Theorem~\ref{thm:precond:circ} could be extended to handle the dual problem by using the diagonal blocks of the dual Hessian in the preconditoner instead.
 Further work is needed to extend the theoretical bounds for the preconditioned Hessian in Section~\ref{sec:precond:analysis} though, since it isn't known if the dual Hessian possesses a block Toeplitz structure like the primal Hessian.

\balance
\bibliography{main}

\end{document}

%% file: figures/spectrum_schur.tex
\begin{figure}[t!]
    \centering
    
    \pgfplotsset{Q_inst/.style={only marks, red, every mark/.append style={solid, fill=white}, mark = *}}
    \pgfplotsset{Q_limit/.style={dashed, very thick, blue}}
    \pgfplotsset{Lyap_inst/.style={only marks, blue, every mark/.append style={solid, fill=white}, mark = triangle}}
    \pgfplotsset{S_inst/.style={only marks, brown, every mark/.append style={solid, fill=white}, mark = square}}
    \pgfplotsset{S_limit/.style={dotted, very thick, brown}}
    
    \subfloat[Original Hessian]
    {
    \label{fig:schur:original}
        \begin{tikzpicture}
            \pgfplotstableread[col sep=comma]{figures/data/spectrum_JonesMorari.csv}{\spectrumData}

            \begin{axis}[xmin   = 0,
                         xmax   = 60,
                         ymin   = 2,
                         ymax   = 12,
                         grid   = major,
                         xlabel = {Horizon Length $N$},
                         ylabel = {$\kappa(H)$},
                         height = 4cm,
                         width  = 0.45\textwidth,
                         legend cell align = left,
                         legend columns = 1,
                         legend style={at={(0.5, 1.05)}, anchor=south, font=\small}]

                    \addplot[Q_inst] table [x=N, y=cond_q] {\spectrumData};
                    \addlegendentry{$H$ with $P=Q$}
                                
                    \addplot[Lyap_inst] table [x=N, y=cond_p] {\spectrumData};
                    \addlegendentry{$H$ with $P$ the solution to \eqref{eq:dlyap}}
                    
                    \addplot[Q_limit] table [x=N, y=cond_bound] {\spectrumData};
                    \addlegendentry{Bound from Corollary~\ref{cor:stableHessEig}}
            \end{axis}
        \end{tikzpicture}
    }
    
    \subfloat[Preconditioned Hessian]
    {
        \label{fig:schur:pedcond}
        \begin{tikzpicture}
            \pgfplotstableread[col sep=comma]{figures/data/spectrum_precond_JonesMorari.csv}{\precondData}
    
            \begin{axis}[xmin   = 0,
                         xmax   = 60,
                         ymin   = 1,
                         ymax   = 4,
                         grid   = major,
                         xlabel = {Horizon Length $N$},
                         ylabel = {$\kappa(H_{L})$},
                         height = 4cm,
                         width  = 0.45\textwidth,
                         legend cell align = left,
                         legend columns = 1,
                         legend style={at={(0.5, 1.05)}, anchor=south, font=\small}]

                    \addplot[Q_inst] table [x=N, y=cond_sdp] {\precondData};
                    \addlegendentry{SDP \cite{Richter2012_FGMcomplexity}}

                    \addplot[Lyap_inst] table [x=N, y=cond_strang] {\precondData};
                    \addlegendentry{Proposed (Theorem~\ref{thm:precond:circ})}
                    
                    \addplot[Q_limit] table [x=N, y=bound_strang] {\precondData};
                    \addlegendentry{Bound from Section~\ref{sec:precond:analysis}}
            \end{axis}
        \end{tikzpicture}
    }

    \caption{Condition number versus the horizon length of the condensed Hessian for the Schur-stable system~\eqref{eq:schurEx:normal}}
    \label{fig:schur}
\end{figure}
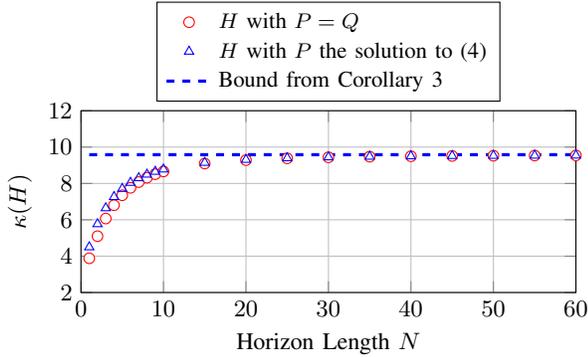
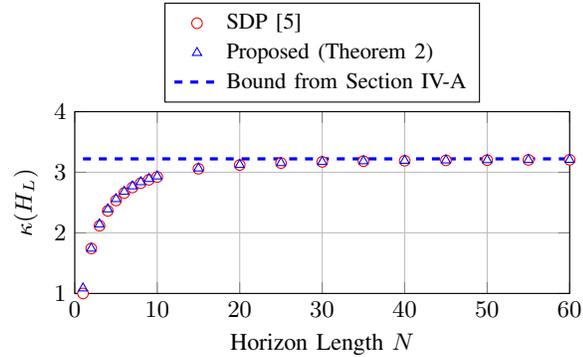

%% file: figures/spectrum_invertedPendulum.tex
\begin{figure}[t!]
    \centering
    
    \pgfplotsset{Q_inst/.style={only marks, red, every mark/.append style={solid, fill=white}, mark = *}}
    \pgfplotsset{Q_limit/.style={dashed, very thick, blue}}
    \pgfplotsset{Lyap_inst/.style={only marks, blue, every mark/.append style={solid, fill=white}, mark = triangle}}
    \pgfplotsset{S_inst/.style={only marks, brown, every mark/.append style={solid, fill=white}, mark = square}}
    \pgfplotsset{S_limit/.style={dotted, very thick, brown}}
    
    \subfloat[Original Hessian]
    {
    \label{fig:pendulum:original}
        \begin{tikzpicture}
            \pgfplotstableread[col sep=comma]{figures/data/spectrum_InvertedPendulum.csv}{\spectrumData}

            \begin{axis}[xmin   = 0,
                         xmax   = 175,
                         ymin   = 1,
                         ymax   = 3,
                         grid   = major,
                         xlabel = {Horizon Length $N$},
                         ylabel = {$\kappa(H)$},
                         height = 4cm,
                         width  = 0.45\textwidth,
                         legend cell align = left,
                         legend columns = 1,
                         legend style={at={(0.5, 1.05)}, anchor=south, font=\small}]
                               
                    \addplot[Lyap_inst] table [x=N, y=cond_p] {\spectrumData};
                    \addlegendentry{$H$ with $P$ the solution to \eqref{eq:dare}}
                    
                    \addplot[Q_limit] table [x=N, y=cond_bound] {\spectrumData};
                    \addlegendentry{Bound from Theorem~\ref{thm:dense:hessEig:condLimit}}
            \end{axis}
        \end{tikzpicture}
    }
    
    \subfloat[Preconditioned Hessian]
    {
        \label{fig:pendulum:pedcond}
        \begin{tikzpicture}
            \pgfplotstableread[col sep=comma]{figures/data/spectrum_precond_InvertedPendulum.csv}{\precondData}
    
            \begin{axis}[xmin   = 0,
                         xmax   = 100,
                         ymin   = 1,
                         ymax   = 2.6,
                         grid   = major,
                         xlabel = {Horizon Length $N$},
                         ylabel = {$\kappa(H_{L})$},
                         height = 4cm,
                         width  = 0.45\textwidth,
                         legend cell align = left,
                         legend columns = 1,
                         legend style={at={(0.5, 1.05)}, anchor=south, font=\small}]

                    \addplot[S_inst] table [x=N, y=cond_orig] {\precondData};
                    \addlegendentry{Un-preconditioned}
               
                    \addplot[Q_inst] table [x=N, y=cond_sdp] {\precondData};
                    \addlegendentry{SDP \cite{Richter2012_FGMcomplexity}}
               
                    \addplot[Lyap_inst] table [x=N, y=cond_strang] {\precondData};
                    \addlegendentry{Proposed (Theorem~\ref{thm:precond:circ})}
                    
                    \addplot[Q_limit] table [x=N, y=bound_strang] {\precondData};
                    \addlegendentry{Bound from Section~\ref{sec:precond:analysis}}
            \end{axis}
        \end{tikzpicture}
    }

    \caption{Condition number versus the horizon length of the pre-stabilized condensed Hessian for the inverted pendulum system.}
    \label{fig:pendulum}
\end{figure}
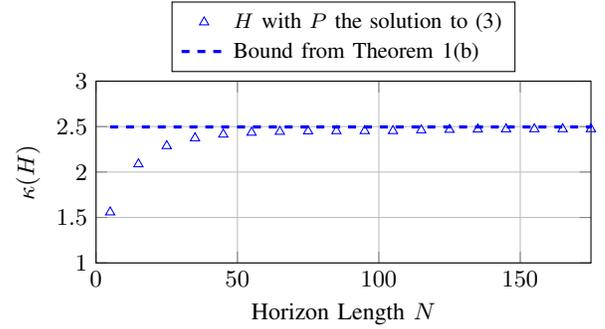
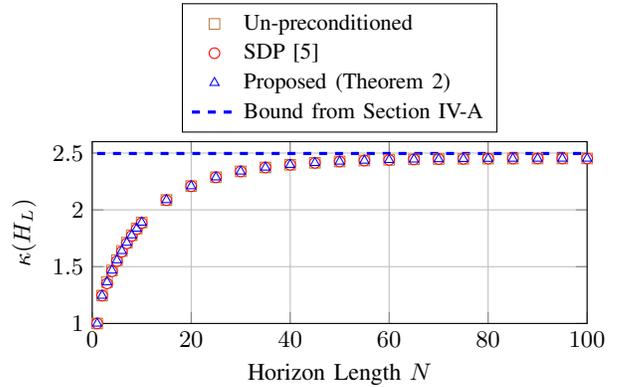